\documentclass[12pt]{amsart}
\usepackage{amssymb,amsmath,amsfonts,latexsym}
\usepackage{bm, enumerate}

\newcommand{\newsection}[1]{\setcounter{equation}{0} \section{#1}}
\newcommand{\bea}{\begin{eqnarray}}
\newcommand{\eea}{\end{eqnarray}}

\newcommand{\clb}{\mathcal{B}}

\newcommand{\cle}{\mathcal{E}}

\newcommand{\clh}{\mathcal{H}}

\newcommand{\clm}{\mathcal{M}}

\newcommand{\cls}{\mathcal{S}}

\newcommand{\clw}{\mathcal{W}}

\newcommand{\D}{\mathbb{D}}

\newcommand{\C}{\mathbb{C}}

\newcommand{\raro}{\rightarrow}

\def\textmatrix#1&#2\\#3&#4\\{\bigl({#1 \atop #3}\ {#2 \atop #4}\bigr)}
\def\dispmatrix#1&#2\\#3&#4\\{\left({#1 \atop #3}\ {#2 \atop #4}\right)}
\newcommand{\be}{\begin{equation}}
\newcommand{\ee}{\end{equation}}
\newcommand{\ben}{\begin{eqnarray*}}
\newcommand{\een}{\end{eqnarray*}}

\newcommand{\NI}{\noindent}

\newcommand{\bi}{\begin{itemize}}
\newcommand{\ei}{\end{itemize}}

\newcommand\vp{\varphi}

\newtheorem{Theorem}{\sc Theorem}[section]
\newtheorem{Lemma}[Theorem]{\sc Lemma}
\newtheorem{Proposition}[Theorem]{\sc Proposition}
\newtheorem{Corollary}[Theorem]{\sc Corollary}
\newtheorem{Example}[Theorem]{\sc Example}
\newtheorem{Remark}[Theorem]{\sc Remark}

\newtheorem{Note}[Theorem]{\sc Note}
\newtheorem{Question}{\sc Question}
\newtheorem{ass}[Theorem]{\sc Assumption}
\newtheorem{Definition}[Theorem]{\sc Definition}
\newcommand{\bt}{\begin{Theorem}}
\def\beginlem{\begin{Lemma}}
\def\beginprop{\begin{Proposition}}
\def\begincor{\begin{Corollary}}
\def\begindef{\begin{Definition}}
\def\beginexamp{\begin{Example}}
\def\beginrem{\begin{Remark}}
\def\beginq{\begin{Question}}
\def\beginass{\begin{ass}}
\def\beginnote{\begin{Note}}
\newcommand{\et}{\end{Theorem}}
\def\endlem{\end{Lemma}}
\def\endprop{\end{Proposition}}
\def\endcor{\end{Corollary}}
\def\enddef{\end{Definition}}
\def\endexamp{\end{Example}}
\def\endrem{\end{Remark}}
\def\endq{\end{Question}}
\def\endass{\end{ass}}
\def\endnote{\end{Note}}

\begin{document}

\title{Invariant subspaces of analytic perturbations}

\author[Das]{Susmita Das}
\address{Indian Statistical Institute, Statistics and Mathematics Unit, 8th Mile, Mysore Road, Bangalore, 560059,
India}
\email{susmita.das.puremath@gmail.com}

\author[Sarkar]{Jaydeb Sarkar}
\address{Indian Statistical Institute, Statistics and Mathematics Unit, 8th Mile, Mysore Road, Bangalore, 560059,
India}
\email{jay@isibang.ac.in, jaydeb@gmail.com}

%\today



\begin{abstract}
By analytic perturbations, we refer to shifts that are finite rank perturbations of the form $M_z + F$, where $M_z$ is the unilateral shift and $F$ is a finite rank operator on the Hardy space over the open unit disc. Here shift refers to the multiplication operator $M_z$ on some analytic reproducing kernel Hilbert space. In this paper, we first isolate a natural class of finite rank operators for which the corresponding perturbations are analytic, and then we present a complete classification of invariant subspaces of those analytic perturbations. We also exhibit some instructive examples and point out several distinctive properties (like cyclicity, essential normality, hyponormality, etc.) of analytic perturbations.
\end{abstract}

\maketitle

\tableofcontents

\newsection{Introduction}\label{sec: intro}

Perturbation theory for linear operators is an old subject that studies spectral theory and the structural behavior of linear operators that are perturbed by small operators (see the classic \cite{Kato}). Broadly speaking, the main aim of perturbation theory is to study (and also compare the properties of)
\[
S = T + F,
\]
where $T$ is a tractable operator (like unitary, normal, isometry, self-adjoint, etc.) and $F$ is a finite rank (or compact, Hilbert–Schmidt, Schatten-von Neumann class, etc.) operator on some Hilbert space.

The theory of perturbed linear operators is far from complete and there are many open problems and untouched areas (cf. \cite{Clark, Fuhrmann, Nakamura, PY, Simon}). In this note, however, we propose an analytic approach to perturbation theory, namely, we study analytic perturbations of unilateral shift on the Hardy space $H^2(\D)$ over the open unit disc $\D$ in $\mathbb{C}$. More specifically, we deal with closed invariant subspaces of ``shift'' operators of the form
\[
S_n = M_z + F,
\]
where $M_z$ denotes the unilateral shift and $F$ is a finite rank operator (of rank $\leq n$) on $H^2(\D)$. We call a bounded linear operator $S$ acting on a Hilbert space a \textit{shift} if $S$ is unitarily equivalent to $M_z$ on some \textit{analytic Hilbert space}, where $M_z$ denote the multiplication operator by the coordinate function $z$. In this paper, analytic Hilbert spaces will refer to reproducing kernel Hilbert spaces of analytic functions on $\D$. The unilateral shift $M_z$ on $H^2(\D)$ is a natural example (which is also a model example of isometry) of shift.

Now the classification of invariant subspaces of the unilateral shift is completely known, thanks to the classical work of Beurling \cite{Beurling}: A nonzero closed subspace $\clm \subseteq H^2(\D)$ is invariant under $M_z$ if and only if there exists an inner function $\theta \in H^\infty(\D)$ such that
\[
\clm = \theta H^2(\D).
\]
We use the standard notation $H^\infty(\D)$ to denote the Banach algebra of all bounded analytic functions on $\D$.

In this paper, we first introduce a class of finite rank operators $F$ (we call them \textit{$n$-perturbations}) on $H^2(\D)$ for which the corresponding perturbations $S_n = M_z + F$ are shifts (we call them \textit{$n$-shifts}). Then we present a complete classification of $S_n$-invariant closed subspaces of $H^2(\D)$. Note again that $S_n$ is unitarily equivalent to the multiplication operator $M_z$ on some analytic Hilbert space.

Our central result (see Theorem \ref{thm: inv sub}) is the following invariant subspace theorem (see Definition \ref{def: n shift} for the formal definition of $n$-shifts): Let $S_n = M_z + F$ on $H^2(\D)$ be an $n$-shift, and let $\clm$ be a nonzero closed subspace of $H^2(\D)$. Then $\clm$ is invariant under $S_n$ if and only if there exist an inner function $\theta \in H^\infty(\D)$ and polynomials $\{p_i, q_i\}_{i=0}^{n-1} \subseteq \C[z]$ such that
\[\clm = (\mathbb{C}\vp_0 \oplus \mathbb{C}\vp_1 \oplus \cdots \oplus \mathbb{C}\vp_{n-1}) \oplus z^n\theta H^2(\D),\]
where $\vp_i = z^ip_i \theta - q_i$ for all $i=0, \ldots, n-1$, and
\[
S_n\vp_j \in (\mathbb{C}\vp_{j+1} \oplus \cdots \oplus \mathbb{C}\vp_{n-1})\oplus z^n\theta H^2(\D),
\]
for all $j=0, \ldots, n-2$, and $S_n\vp_{n-1} = z^n p_{n-1} \theta$.

The above classification is based on a result of independent interest (see Theorem \ref{prop: dimension 1}): If $\clm$ is a nonzero closed $S_n$-invariant subspace of $H^2(\D)$, then
\[
\text{dim} (\clm \ominus S_n \clm) = 1.
\]
Clearly, this is a Burling-type property of $S_n$-invariant subspaces.

We remark that a priori examples of $n$-shifts may seem counter-intuitive because of the intricate structure of perturbed of linear operators. Subsequently, we put special emphasis on natural examples of $n$-shifts, and as interesting as it may seem, analytic spaces corresponding to (truncated) tridiagonal kernels or band kernels with bandwidth $1$ give several natural examples of $n$-shifts. In the special case when $S_n$ is unitarily equivalent to a shift on an analytic space corresponding to a band truncated kernel with bandwidth $1$, we prove that the invariant subspaces of $S_n$ are also hyperinvariant. Our proof of this fact follows a classical route: computation of commutants of shifts. In general, it is a difficult problem to compute the commutant of a shift (even for weighted shifts). However, in our band truncated kernel case, we are able to explicitly compute the commutant of $n$-shifts:
\[
\{S_n\}' = \{T_{\vp} + N: \vp \in H^\infty(\D), \text{rank} N \leq n\},
\]
where $T_{\vp}$ denotes the analytic Toeplitz operator with symbol $\vp \in H^\infty(\D)$, and $N$ admits an explicit (and restricted) representation (cf. \eqref{eqn: matrix M1}). We also present concrete examples of $1$-shifts on tridiagonal kernel spaces with special emphasis on cyclicity of invariant subspaces. For instance, a simple example of $S_1$-shift in Section \ref{sec: examples} brings out the following distinctive properties:
\begin{enumerate}
\item $[S_1^*, S_1] := S_1^* S_1 - S_1 S_1^*$ is of finite rank (in particular, $S_1$ is essentially normal).
\item $S_1$ is not subnormal (and, more curiously, not even hyponormal).
\item Invariant subspaces of $S_1$ are cyclic.
\end{enumerate}

We believe that these observations along with the classification of invariant subspaces of shifts on tridiagonal spaces (a particular case of Theorem \ref{thm: inv sub}) are of independent interest beside their application to the theory of perturbed operators. Finally, we remark that perturbations of concrete operators (with some analytic flavor) have been also studied in different contexts by other authors. For instance, see \cite{ALP, Fuhrmann, Nakamura, AP, PY, Serban}, and notably Clark \cite{Clark}.

The rest of this paper is organized as follows: In Section \ref{sec: n shift} we formally introduce $n$-perturbations and $n$-shifts, and collect all the necessary preliminaries about $n$-shifts. Section \ref{sec: main thm} deals with the invariant subspace theorem of $n$-shifts.

\NI In Section \ref{sec: commutant}, we restrict our study to $n$-shifts on truncated tridiagonal spaces. We remark that shifts on tridiagonal spaces are the ``next best'' examples of shifts after the weighted shifts. In this case, we completely parameterize the commutants of $n$-shifts. In particular, we prove that the multiplier space of a truncated tridiagonal space is precisely $H^\infty(\D)$.

\NI In Section \ref{sec: Hyper}, we use the structure of commutants of shifts on truncated tridiagonal spaces to prove that the invariant subspaces of $n$-shifts are actually hyperinvariant. The final section, Section \ref{sec: examples}, is devoted to instructive examples. Here we illustrate the main result, Theorem \ref{thm: inv sub}, with some concrete examples, and present a classification of cyclic invariant subspaces of $1$-shifts.

In this paper, all Hilbert spaces will be separable and over $\C$. Given a Hilbert space $\clh$, $\clb(\clh)$ will denote the algebra of all bounded linear operators on $\clh$. Throughout this paper, $n$ will be an arbitrary natural number.

\newsection{$n$-shifts}\label{sec: n shift}

This section introduces the central concept of this paper, namely, analytic perturbations or $n$-shifts. We also explore some basic properties of $n$-shifts.  

We begin with a concise discussion of shift operators. Briefly speaking, a shift operator is the multiplication operator $M_z$ by the coordinate function $z$ on some Hilbert space of analytic functions on a domain in $\C$. More specifically, given a Hilbert space $\cle$, a function $k : \D \times \D \raro \clb(\cle)$ is called \textit{positive definite} or a \textit{kernel} \cite{Aronszajn} if
\begin{equation}\label{eqn: kernel vec}
\sum_{i,j=1}^m \langle k(z_i, z_j) \eta_j, \eta_i \rangle_{\cle} \geq 0,
\end{equation}
for all $\{z_1, \ldots, z_m\} \subseteq \D$, $\{\eta_1, \ldots, \eta_m\} \subseteq \cle$ and $m \geq 1$. A kernel $k$ is called \textit{analytic} if $k$ is analytic in the first variable. As is well known, if $k$ is an analytic kernel, then there exists a Hilbert space $\clh_k$, which we call \textit{analytic Hilbert space}, of $\cle$-valued analytic functions on $\D$ such that $\{k(\cdot, w) \eta: w \in \D, \eta \in \cle\}$ is a total set in $\clh_k$ with the \textit{reproducing property}
\[
\langle f(w), \eta \rangle_{\cle} = \langle f, k(\cdot, w) \eta \rangle_{\clh_k},
\]
for all $f \in \clh_k$, $w \in \D$, and $\eta \in \cle$. We now present the formal definition of shift operators:
\begin{Definition}
The shift on $\clh_k$ is the multiplication operator $M_z$ defined by $(M_z f)(w) = w f(w)$ for all $f \in \clh_k$ and $w \in \D$.
\end{Definition}

In what follows, we will be mostly concerned with bounded shifts. Therefore, we always assume that $M_z$ is bounded. Note that, in the scalar-valued case, that is, when $\cle = \C$, the positivity condition in \eqref{eqn: kernel vec} becomes
\[
\sum_{i,j=1}^m \bar{c}_i c_j k(z_i, z_j) \geq 0,
\]
for all $\{z_1, \ldots, z_m\} \subseteq \D$, $\{\eta_1, \ldots, \eta_m\} \subseteq \cle$ and $m \geq 1$. The simplest example of an analytic kernel is the \textit{Szeg\"{o} kernel} $\mathbb{S}$ on $\D$, where
\[
\mathbb{S}(z, w) = (1 - z \bar{w})^{-1} \qquad (z, w \in \D).
\]
The analytic space corresponding to the Szeg\"{o} kernel is the well-known (scalar-valued) Hardy space $H^2(\D)$, where the shift $M_z$ on $H^2(\D)$ is known as the \textit{unilateral shift} (of multiplicity one). Also, recall that the unilateral shift $M_z$ on $H^2(\D)$ is the model operator for contractions on Hilbert spaces (in the sense of basic building blocks \cite{Clark}).

We also record the key terms of the agreement: $X_1 \in \clb(\clh_1)$ and $X_2 \in \clb(\clh_2)$ are the same means there exists a unitary $U: \clh_1 \raro \clh_2$ such that $UX_1 = X_2 U$, that is, $X_1$ and $X_2$ are unitarily equivalent. Therefore, $X \in \clb(\clh)$ is a shift if there exists an analytic Hilbert space $\clh_k$ such that the shift $M_z$ on $\clh_k$ and $X$ are unitarily equivalent. Finally, we are ready to introduce the central objects of this paper:

\begin{Definition}[$n$-shifts]\label{def: n shift}
A linear operator $F$ on $H^2(\D)$ is called an $n$-perturbation if

(i) $Fz^m = 0$ for all $m \geq n$,

(ii) $F (z^m H^2(\D)) \subseteq z^{m+1} \C[z]$ for all $m \geq 0$, and

(iii) $M_z+F$ is left-invertible.

\NI We call $S_n = M_z + F$ the $n$-shift corresponding to the $n$-perturbation $F$ (or simply $n$-shift if $F$ is clear from the context).
\end{Definition}

Since $\text{ran} F \subseteq \text{span} \{1, z, \ldots, z^{n-1}\}$, it follows that an $n$-perturbation is of rank $m$ for some $m \leq n$. In fact, it is easy to see that the rank of the $2$-perturbation
\[
Fz^m = \begin{cases}
z^2 & \mbox{if } m =0,1  \\
0 & \mbox{otherwise},
\end{cases}
\]
is precisely $1$. Moreover, $S_2 = M_z + F$ is a $2$-shift. Indeed, since $S_2^* S_2 = \begin{bmatrix} 2 & 2 \\ 2 & 4 \end{bmatrix} \oplus I_{z^2 H^2(\D)}$ on $H^2(\D) = \C \oplus \C z \oplus z^2 H^2(\D)$, it follows that $S_2^* S_2$ is invertible, and hence $S_2$ is left-invertible. Now we justify Definition \ref{def: n shift} by showing that an $n$-shift is indeed a shift.

\begin{Lemma}\label{lemma: S_n is shift}
Let $F$ be an $n$-perturbation. If $S_n = M_z + F$, then:

(i) $F(z^m f) = 0$ for each $m \geq n$ and $f \in H^2(\D)$.

(ii) For each $f \in H^2(\D)$ and $m \geq 1$, there exists $p \in \C[z]$, depending on both $f$ and $m$, such that
\[
S_n^mf = z^m (f + p).
\]

(iii) $S_n$ is a shift on some analytic Hilbert space.
\end{Lemma}
\begin{proof}
Part (i) immediately follows from the fact that $F(z^m p) = 0$ for all $p \in \C[z]$. Since by assumption $F (z^m H^2(\D)) \subseteq z^{m+1} \C[z]$, $m \geq 0$, for each $f \in H^2(\D)$, there exists a polynomial $p_f \in \C[z]$ such that $F f = z p_f$. Then
\[
S_n f = (M_z + F)f = zf + zp_f = z(f+p_f),
\]
and hence, there exists $q_f \in \C[z]$ such that
\[
S_n^2f = (M_z + F) (z(f + p_f)) = z^2(f + p_f) + z^2 q_f = z^2(f + p_f + q_f).
\]
The second assertion now follows by the principle of mathematical induction. To prove the last assertion, we use (ii) to conclude that
\begin{equation}\label{eqn:Sm H subset zm H}
S_n^m H^2(\D) \subseteq z^m H^2(\D) \qquad (m \geq 0).
\end{equation}
Since we know that $M_z$ on $H^2(\D)$ is pure, that is, $\cap_{m \geq 0} z^m H^2(\D) = \{0\}$, the above inclusion implies that
\[
\cap_{m \geq 0} S_n^m H^2(\D) \subseteq \cap_{m \geq 0} z^m H^2(\D) = \{0\}.
\]
Using this and the left invertibility of $S_n$, it follows that $S_n$ on $H^2(\D)$ is a shift.
\end{proof}

Note that the following standard fact \cite{SS} has been used in the above proof: If $T \in \clb(\clh)$ is a left-invertible operator and if $\cap_{m=0}^\infty T^m \clh = \{0\}$, then $T$ is unitarily equivalent to the shift $M_z$ on some $\clw$-valued analytic Hilbert space, where $\clw = \clh \ominus T \clh$. In the present case, if
\[
\clw = \ker S_n^* = \ker (M_z + F)^*,
\]
then $S_n$ on $H^2(\D)$ is unitarily equivalent to $M_z$ on some $\clw$-valued analytic Hilbert space $\clh_k$ over $\D$. Here the kernel function $k$ is explicit \cite[ Corollary 2.14]{SS} and involves a specific left inverse of $S_n$ (namely, $(S_n^*S_n)^{-1}S_n^*$), but we will not need this.

Let $T$ be a bounded linear operator on a Hilbert space $\clh$. Given a vector $f \in \clh$, let $[f]_{T}$ denote the $T$-cyclic closed subspace generated by $f$, that is
\[
[f]_T = \text{clos } \{p(T) f: p \in \C[z]\}.
\]

\begin{Lemma}\label{lemma: [f] cont shift}
If $f \in H^2(\D)$ is a nonzero function, then $[f]_{S_n}$ contains a nontrivial closed $M_z$-invariant subspace of $H^2(\D)$.
\end{Lemma}
\begin{proof}
Suppose $g \in H^2(\D)$. By part (ii) of Lemma \ref{lemma: S_n is shift}, we already know that $S_n^ng = z^n(g + p)$ for some $p \in \C[z]$. Then part (i) of the same lemma implies that
\[
S_n^{n+1}g = (M_z + F) (z^n g + z^n p) = M_z(z^{n} g + z^n p) = M_z(S_n^n g).
\]
Then, by induction, we have $S_n^m g = M_z^{m-n} (S_n^n g)$, and hence
\begin{equation}\label{eqn: Smf= Mzlm}
S_n^m = M_z^{m-n} S_n^n  \qquad (m \geq n+1).
\end{equation}
In particular, if $f$ is nonzero in $H^2(\D)$, then $[S_n^n f]_{M_z}$ is an $M_z$-invariant closed subspace of $[f]_{S_n}$.
\end{proof}

In the context of the equality \eqref{eqn: Smf= Mzlm}, note in general that
\[
[M_z^{m-n}, S_n^n] = M_z^{m-n} S_n^n - S_n^n M_z^{m-n} \neq 0 \qquad (m \geq n+1).
\]

\newsection{Invariant subspaces}\label{sec: main thm}

In this section, we will prove the central result of this paper: a complete classification of $n$-shift invariant closed subspaces of $H^2(\D)$. However, as a first step, we need to prove a Beurling type property of invariant subspaces of $n$-shifts. We recall that if $\cls$ is a nonzero closed $M_z$-invariant subspace of $H^2(\D)$, then
\[
\text{dim } (\cls \ominus z \cls) = 1.
\]
This is an easy consequence of the Beurling theorem (or, one way to prove the Beurling theorem). In the following, we prove a similar result for $S_n$-invariant closed subspaces of $H^2(\D)$.

\begin{Theorem}\label{prop: dimension 1}
If $\clm \subseteq H^2(\D)$ is a nonzero closed $S_n$-invariant subspace, then
\[
\text{dim}(\clm \ominus S_n \clm) = 1.
\]
\end{Theorem}
\begin{proof}
Suppose if possible that $\clm \ominus S_n \clm = \{0\}$. Since $S_n$ is left-invertible, it follows that
\[
S_n^m \clm = \clm \qquad (m \geq 1),
\]
which implies that
\[
\clm = \cap_{m\geq 1} S_n^m \clm \subseteq \cap_{m\geq 1} S_n^m H^2(\D) \subseteq \cap_{m\geq 1} z^m H^2(\D) = \{0\},
\]
where the second inclusion follows from \eqref{eqn:Sm H subset zm H}. This contradiction shows that $\clm \ominus S_n \clm \neq \{0\}$. Now suppose that $f, g \in \clm \ominus S_n \clm$ be unit vectors. If possible, assume that $f$ and $g$ are orthogonal, that is, $\langle f, g \rangle = 0$. We claim that
\[
[f]_{S_n} \cap [g]_{S_n} = \{0\}.
\]
To prove this, first we pick a nonzero vector $h \in [f]_{S_n} \cap [g]_{S_n}$. Then there exist sequences of polynomials $\{p_m\}_{m \geq 1}$ and $\{q_m\}_{m\geq 1}$ such that
\begin{equation}\label{eqn: h = lim}
h = \lim_{m \raro \infty} (p_m(S_n) f) = \lim_{m \raro \infty} (q_m(S_n) g).
\end{equation}
For each $m \geq 1$, we let
\[
p_m(z) = \alpha_{m,0} + \alpha_{m,1} z + \cdots + \alpha_{m, t_m} z^{t_m},
\]
and
\[
q_m(z) = \beta_{m, 0} + \beta_{m, 1} z + \cdots + \beta_{m, l_m} z^{l_m},
\]
where $t_m$ and $l_m$ are in $\mathbb{N}$ and $m \geq 1$. Now $S_n^l g \in S_n \clm$ for all $l \geq 1$, together with $\langle g, f \rangle = 0$ implies that $\langle q_m(S_n) g, f \rangle = 0$ for all $m \geq 1$. Therefore
\[
\langle h, f \rangle = \langle \lim_{m \raro \infty} p_m(S_n) f, f \rangle = \langle  \lim_{m \raro \infty} q_m(S_n) g, f \rangle = \lim_{m \raro \infty} \langle q_m(S_n) g, f \rangle = 0,
\]
that is, $\langle h, f \rangle = 0$, where, on the other hand
\[
\langle h, f \rangle = \langle \lim_{m \raro \infty} p_m(S_n) f, f \rangle = \lim_{m \raro \infty} \langle p_m(S_n) f, f \rangle = \lim_{m \raro \infty} \langle \alpha_{m,0} f, f \rangle,
\]
as $S_n^l f \in S_n \clm$ for all $l \geq 1$, and $f \perp S_n \clm$. We immediately deduce that
\[
\lim_{m \raro \infty} \alpha_{m,0} = 0.
\]
Thus we obtain
\[
h= \lim_{m \raro \infty} ((\alpha_{m,1} S_n + \cdots + \alpha_{m, t_m} S_n^{t_m})f).
\]
Since $\langle S_n^k g , g \rangle = 0$ and $\langle S_n^l f, g \rangle = 0 $ for all $ k, l \geq 1$, repeating the same argument as above, we have $\langle h, g \rangle = 0$ and
\[
\lim_{m\raro \infty} \beta_{m,0} = 0,
\]
and consequently
\[
h= \lim_{m\raro \infty} ((\beta_{m,1} S_n + \cdots + \beta_{m, l_m} S_n^{l_m})g).
\]
Thus we obtain
\[
\lim_{m \raro \infty} ((\alpha_{m,1} S_n + \cdots + \alpha_{m, t_m} S_n^{t_m})f) = \lim_{m \raro \infty} ((\beta_{m,1} S_n + \cdots + \beta_{m, l_m} S_n^{l_m})g).
\]
Multiplying both sides by a left inverse of $S_n$ (for instance, $(S_n^* S_n)^{-1} S_n^*$ is a left inverse of $S_n$ \cite{SS}) then gives
\[
\begin{split}
h_1:&  = \lim_{m \raro \infty} ((\alpha_{m,1} + \alpha_{m,2} S_n + \cdots + \alpha_{m, t_m} S_n^{t_m-1})f)
\\
& = \lim_{m\raro \infty} ((\beta_{m,1} + \beta_{m,2} S_n + \cdots + \beta_{m, l_m} S_n^{l_m-1})g).
\end{split}
\]
We are now in exactly the same situation as in \eqref{eqn: h = lim}. Proceeding as above, we then have
\[
\lim_{m\raro \infty} \alpha_{m,1} = \lim_{m\raro \infty} \beta_{m,1} = 0.
\]
Arguing similarly, it will follow by induction that
\[
\lim_{m\raro \infty} \alpha_{m,t} = \lim_{m\raro \infty} \beta_{m,l} = 0.
\]
for all $t = 0, 1, \ldots, t_m$, and  $l = 0, 1, \ldots, l_m$, and $m \geq 1$, and so $h=0$. This contradiction shows that
\[
[f]_{S_n}\cap[g]_{S_n}=\{0\}.
\]
Now by Lemma \ref{lemma: [f] cont shift} and the classical Beurling theorem, we know that $\theta_1 H^2(\D) \subseteq [f]_{S_n}$ and $\theta_2 H^2(\D) \subseteq [g]_{S_n}$ for some inner functions $\theta_1$ and $\theta_2$ in $H^\infty(\D)$. Since
\[
\theta_1 \theta_2 \in \theta_1 H^2(\D) \cap \theta_2 H^2(\D) \subseteq [f]_{S_n}\cap[g]_{S_n},
\]
it follows that $\theta_1 H^2(\D) \cap \theta_2 H^2(\D) \neq \{0\}$, which contradicts the fact that $[f]_{S_n}\cap[g]_{S_n}=\{0\}$. Therefore, $\text{dim} (\clm \ominus S_n \clm) = 1$, and completes the proof of the theorem.
\end{proof}

Note that the final part of the above proof uses the classical Beurling theorem  (see the first part of Section \ref{sec: intro}): If $\clm$ is a nonzero $M_z$-invariant closed subspace of $H^2(\D)$, then there exists an inner function $\theta \in H^\infty(\D)$ such that $\clm = [\theta]_{M_z}$. We will return to the issue of cyclic invariant subspaces of $1$-shifts in Section \ref{sec: examples}, and here we proceed to state and prove our general invariant subspace theorem.

\begin{Theorem}\label{thm: inv sub}
Let $F$ be an $n$-perturbation on $H^2(\D)$, and let $\clm$ be a nonzero closed subspace of $H^2(\D)$. Then $\clm$ is invariant under $S_n = M_z + F$ if and only if there exist an inner function $\theta \in H^\infty(\D)$ and polynomials $\{p_i, q_i\}_{i=0}^{n-1} \subseteq \C[z]$ such that
\[\clm = (\mathbb{C}\vp_0 \oplus \mathbb{C}\vp_1 \oplus \cdots \oplus \mathbb{C}\vp_{n-1}) \oplus z^n\theta H^2(\D),\]
where $\vp_i = z^ip_i \theta - q_i$ for all $i=0, \ldots, n-1$, and
\[
S_n\vp_j \in (\mathbb{C}\vp_{j+1} \oplus \cdots \oplus \mathbb{C}\vp_{n-1})\oplus z^n\theta H^2(\D),
\]
for all $j=0, \ldots, n-2$, and $S_n\vp_{n-1} = z^n p_{n-1} \theta$.
\end{Theorem}

\begin{proof}
Let $\clm$ be a nonzero closed subspace of $H^2(\D)$. Observe that
\[
S_n (z^nf) = (M_z+ F) (z^nf) = z^{n+1} f + F(z^nf) = z^{n+1} f,
\]
for all $f \in H^2(\D)$, where the last equality follows from Lemma \ref{lemma: S_n is shift}. Therefore
\begin{equation}\label{eqn:S M = M}
M_z^{m+n} = S_n^m M_z^n \qquad (m \geq 1).
\end{equation}
To prove the sufficient part, we see, by \eqref{eqn:S M = M}, that
\[
S_n(z^n\theta f)=z^{n+1}\theta f \in z^n\theta H^2(\D),
\]
for all $f\in H^2(\D)$, and hence $S_n (z^n\theta H^2(\D)) \subseteq z^n\theta H^2(\D)$. This and the remaining assumptions then implies that $S_n \clm \subseteq \clm$.

\NI For the converse direction, assume that $S_n \clm \subseteq \clm$. Theorem \ref{prop: dimension 1} then implies
\[
\clm = \mathbb{C}\vp_0\oplus S_n\clm,
\]
for some nonzero vector $\vp_0 \in \clm \ominus S_n \clm$.
Since $\clm$ is closed and $S_n$ is left invertible, it follows that $S_n \clm$ is also a nonzero closed $S_n$-invariant subspace of $H^2(\D)$. By Theorem \ref{prop: dimension 1} again, we have
\[
\clm= \mathbb{C}\vp_0 \oplus(\mathbb{C}\vp_1 \oplus S_n^2\clm),
\]
for some nonzero vector $\vp_1 \in S_n \clm \ominus S_n (S_n \clm)$. Continuing exactly in the same way, by induction, we find $\vp_i \in S_n^i\clm \ominus S_n^{i+1}\clm$, $i=0, 1, \ldots, n-1$, such that
\[
\clm = (\mathbb{C}\vp_0\oplus\mathbb{C}\vp_1\oplus \cdots \oplus \mathbb{C}\vp_{j-1})\oplus S_n^j \clm,
\]
for all $j=1, \ldots, n$. In particular, $\clm= (\mathbb{C}\vp_0\oplus\mathbb{C}\vp_1\oplus \cdots \oplus \mathbb{C}\vp_{n-1})\oplus S_n^n \clm$. Now, by \eqref{eqn: Smf= Mzlm}, we have $M_z(S^n_nf) = S_n^{n+1} f$, $f \in \clm$, which implies that $M_z(S_n^n \clm) \subseteq S_n^n \clm$, that is, $S_n^n \clm$ is a closed nonzero $M_z$-invariant subspace of $H^2(\D)$. By the Beurling theorem this implies that $S_n^n \clm = \tilde{\theta} H^2(\D)$ for some inner function $\tilde \theta \in H^\infty(\D)$. Since each element in $S_n^n \clm $ has a zero of order at least $n$ at $z = 0$ (see part (ii) of Lemma \ref{lemma: S_n is shift}), it follows that $\tilde \theta = z^n \theta$ for some inner function $\theta \in H^\infty(\D)$. Thus
\begin{equation}\label{eqn:SnM = zn theta H2D}
S_n^n \clm = z^n \theta H^2(\D),
\end{equation}
and hence
\[
\clm= (\mathbb{C}\vp_0\oplus\mathbb{C}\vp_1\oplus \cdots \oplus \mathbb{C}\vp_{n-1})\oplus z^n\theta H^2(\D),
\]
for some inner function $\theta \in H^\infty(\D)$. Fix an $i \in \{0, 1, \ldots, n-1\}$. Since $\vp_i \in S_n^i \clm \ominus S_n^{i+1}\clm$, by construction, we have $\vp_i \in S_n^i \clm$, and hence \eqref{eqn:SnM = zn theta H2D} implies
\[
S^{n-i} \vp_i \in S_n^n \clm = z^n \theta H^2(\D).
\]
Therefore, there exists $h_i \in H^2(\D)$ such that
\begin{equation}\label{eqn: Sn phi= z n theta h}
S_n^{n-i} \vp_i = z^n \theta h_i.
\end{equation}
By part (ii) of Lemma \ref{lemma: S_n is shift}, there exists a polynomial $q_i \in \C[z]$ such that $S_n^{n-i} \vp_i = z^{n-i}(\vp_i + q_i)$. Then
\begin{equation}\label{eqn: phi + q = z theta h}
\vp_i + q_i = z^i \theta h_i.
\end{equation}
Since $\vp_i \perp S_n^n \clm = z^n \theta H^2(\D)$, by construction, for each $l \geq 0$, we have
\[
\langle z^i \theta h_i, z^{n+l} \theta \rangle = \langle \vp_i + q_i, z^{n+l} \theta \rangle = \langle q_i, z^{n+l} \theta \rangle,
\]
which, along with $\langle z^i \theta h_i, z^{n+l} \theta \rangle = \langle h_i, z^{n+l - i} \rangle$, implies that
\[
\langle h_i, z^{n+l - i} \rangle = \langle q_i, z^{n+l} \theta \rangle.
\]
Finally, using the fact that $q_i$ is a polynomial, we conclude that for each $i=0,\ldots, n-1$, there exists a natural number $n_i$ such that $\langle h_i, z^t \rangle = 0$ for all $t \geq n_i$, and hence $p_i:= h_i$ is a polynomial. This completes the proof.
\end{proof}

From the final part of the above proof, we note that $h_i:=p_i$ is a polynomial. Therefore, by \eqref{eqn: Sn phi= z n theta h} and \eqref{eqn: phi + q = z theta h}, there exist polynomials $p_i, q_i \in \C[z]$ such that $\vp_i = z^i p_i \theta - q_i$, and
\begin{equation}\label{eqn: phi + q = z theta p}
S_n^{n-i} \vp_i = z^n p_i \theta \qquad (i=0, 1, \ldots, n-1).
\end{equation}

The description of invariant subspaces of $S_n$ as in the above theorem appears to be satisfactory and complete. However, a more detailed delicacy is hidden in the structure of polynomials $\{p_i, q_i\}_{i=0}^{n-1}$ and the finite rank operator $F$. In fact, without much control over these polynomials (and/or the finite rank operator $F$), hardly much can be said about the other basic properties of $n$-shift invariant subspaces. For instance:
\begin{center}
When an $n$-shift invariant subspace is cyclic?
\end{center}
Needless to say, the cyclicity property of shift operators is a complex problem. We will return to this issue in Section \ref{sec: examples}, and refer \cite{ALP, Sola et al} for some modern development of cyclic vectors of shift invariant subspaces of function Hilbert spaces.

\newsection{Commutants}\label{sec: commutant}

In this section, we compute commutants of $n$-shifts on analytic Hilbert spaces corresponding to truncated tridiagonal kernels. The concept of tridiagonal kernels or band kernels with bandwidth one in the context of analytic Hilbert spaces was introduced in \cite{Adam 2001, Paulsen 92}. Note that shifts on analytic Hilbert spaces corresponding to band kernels with bandwidth one are the next best examples of shifts after the weighted shifts.

The following definition is a variant of truncated tridiagonal kernels which is also motivated by a similar (but not exactly the same) concept of kernels in the context of Shimorin's analytic models \cite{DS}.

\begin{Definition}\label{def: trunc}
Let $\clh_k$ be an analytic Hilbert space corresponding to an analytic kernel $k: \D \times \D \raro \C$. We say that $\clh_k$ is a truncated space (and $k$ is a truncated kernel) if:

(i) $\C[z] \subseteq \clh_k$,

(ii) the shift $M_z$ is bounded on $\clh_k$, and

(iii) $\{f_m\}_{m \geq 0}$ forms an orthonormal basis of $\clh_k$, where $f_m = (a_m + b_mz) z^m$, $m \geq 0$, for some scalars $\{a_m\}_{m\geq 0}$ and $\{b_m\}_{m\geq 0}$ such that $a_s \neq 0$ for all $s \geq 0$, and $b_t = 0$ for all $t \geq n$.
\end{Definition}

Note that in the above definition, $n$ is a fixed natural number. Also, in this case, the kernel function $k$ is given by
\[
k(z, w) = \sum_{m=0}^{\infty} f_m(z) \overline{f_m(w)} \quad \quad (z, w \in \D).
\]
If, in addition, $\{|\frac{a_m}{a_{m+1}}|\}_{m\geq 0}$ is bounded away from zero, then $M_z$ on $\clh_k$ is left-invertible \cite[Theorem 3.5]{DS}. Clearly, the above representation of $k$ justifies the use of the term tridiagonal kernel.

\textsf{Throughout this section, we will assume that $a_m = 1$ for all $m \geq 0$}. Using the orthonormal basis $\{f_m = (1 + b_mz) z^m\}_{m\geq 0}$ of $\clh_k$, a simple calculation reveals that (cf. \cite[Section 3]{Adam 2001} or \cite[Section 2]{DS})
\begin{equation}\label{eqn: z^m}
z^m = \sum_{t=0}^{\infty} (-1)^t \Big(\prod_{j=0}^{t-1} b_{m+j} \Big) f_{m+t} \qquad (m \geq 0),
\end{equation}
where $\Pi_{j=0}^{-1} x_{m+j} :=1$. Since $b_m = 0$, $m \geq n$, we have $\prod_{j=0}^{t-1} b_{m+j} = 0$ for all $t \geq n+1$. In particular, the above is a finite sum. We set
\begin{equation}\label{eq: c_mp}
c_{m,p}=b_m-b_{m+p},
\end{equation}
for all $m \geq 0$ and $p\geq1$. Clearly, $c_{m,p} = 0$ for all $m \geq n$. Now $M_z f_m = z^{m+1} + b_m z^{m+2}$ implies that
\[
z f_m = f_{m+1} + (b_m - b_{m+1}) z^{m+2} = f_{m+1} + c_{m,1} z^{m+2},
\]
that is, $z f_m  = f_{m+1} + c_{m,1} z^{m+2}$ for all $m \geq 0$. Then \eqref{eqn: z^m} yields
\begin{equation}\label{eq: M_z}
z f_m  = f_{m+1} + c_{m,1} \sum_{t=0}^\infty (-1)^t \Big(\prod_{j=0}^{t-1} b_{m+2+j}\Big) f_{m+2+t} \quad \quad (m\geq 0).
\end{equation}
Since $c_{m,1} = 0$ for all $m \geq n$, as pointed out earlier, it follows that $z f_m = f_{m+1}$ for all $m \geq n$. In particular, the matrix representation of $M_z$ with respect to the orthonormal basis $\{f_m\}_{m \geq 0}$ is given by (also see \cite[Page 729]{Adam 2001})
\[
[M_z] =  \begin{bmatrix}
0& 0& 0& \dots &0 &0 & \dots\\

1& 0& 0& \dots &0 &0 & \dots\\

c_{0,1}& 1 & 0 &\dots & 0 &0 & \dots\\

-c_{0,1} b_2 & c_{1,1} & 1 &\dots& 0 & 0 & \dots\\

c_{0,1} b_2 b_3 & -c_{1,1} b_3 &c_{2,1} &\ddots& 0 & 0 & \dots\\

\vdots & \vdots & \vdots & \ddots & \ddots &\vdots &\vdots\\

0& 0& 0&\dots & c_{n-1,1}& 1 &\ddots\\

0& 0& 0& \dots & 0& 0 &\ddots\\

\vdots & \ddots & \ddots & \ddots & \ddots &\ddots &\ddots
\\
\end{bmatrix}.
\]
We define the \textit{canonical unitary map} $U: \clh_k\longrightarrow H^2(\D)$ by setting $U f_m = z^m$, $m\geq 0$. It then follows that
\begin{equation}\label{eqn:U map}
UM_z=S_n U,
\end{equation}
where $S_n: = M_z + F$ is the $n$-shift corresponding to the $n$-perturbation $F$ on $H^2(\D)$ whose matrix representation with respect to the orthonormal basis $\{z^m\}_{m \geq 0}$ of $H^2(\D)$ is given by
\[
[F] =  \begin{bmatrix}
0& 0& 0& \dots &0 &0 & \dots\\

0& 0& 0& \dots &0 &0 & \dots\\

c_{0,1}& 0 & 0 &\dots & 0 & 0& \dots\\

-c_{0,1} b_2 & c_{1,1} & 0 &\dots& 0 & 0 & \dots\\

c_{0,1} b_2 b_3 & -c_{1,1} b_3 &c_{2,1} &\dots& 0 & 0 & \dots\\

\vdots & \vdots & \vdots & \ddots & \ddots &\vdots &\vdots\\

0& 0& 0&\dots & c_{n-1,1}& 0 &\cdots\\

0& 0& 0& \dots & 0& 0 &\cdots\\

\vdots & \vdots & \vdots & \dots & \dots &\vdots &\vdots\\
\end{bmatrix},
\]

\begin{Definition}
We call $S_n$ the $n$-shift corresponding to the truncated kernel $k$.
\end{Definition}

Now we turn to the commutants of $n$-shifts corresponding to truncated kernels. Since $M_z$ on $\clh_k$ and $S_n$ on $H^2(\D)$ are unitarily equivalent, the problem of computing the commutant of $S_n$ reduces to that of $M_z$.

Let $\clh_k$ be a truncated space. Recall that a function $\vp: \D \raro \C$ is said to be a \textit{multiplier} of $\clh_k$ if $\vp \clh_k \subseteq \clh_k$ \cite{Aronszajn}. We denote by $\clm(\clh_k)$ the set of all multipliers. By the closed graph theorem, a multiplier $\vp \in \clm(\clh_k)$ defines a bounded linear operator $M_\vp$ on $\clh_k$, where
\[
M_\vp f = \vp f \qquad (f \in \clh_k).
\]
We call $M_\vp$ the \textit{multiplication operator} corresponding to $\vp$.

We will use the following notation: If $X \in \clb(\clh)$, then the commutant of $X$, denoted by $\{X\}'$, is the algebra of all operators $T \in \clb(\clh)$ such that $TX = XT$. In the following, we observe that $\{M_z\}' = \{M_{\vp}: \vp \in \clm(\clh_k)\}$. The proof is fairly standard:

\begin{Lemma}
Suppose $A \in \clb(\clh_k)$. Then $A M_z=M_z A$ if and only if there exists $\vp \in \clm(\clh_k)$ such that $A=M_{\vp}$.
\end{Lemma}

\begin{proof}
The ``if'' part is easy. To prove the ``only if'' part, suppose $A M_z = M_z A$ and let $A 1 = \vp$. Clearly, $\vp \in \clh_k$. Since $f_m = (1 + b_m z)z^m$, it follows that
\[
A f_m = A z^m + b_m A z^{m+1} = (z^m + b_m z^{m+1}) A 1 = f_m \vp = \vp f_m,
\]
for all $m \geq 0$. Since $\{f_m\}_{m \geq 0}$ is an orthonormal basis, we have $Af = \vp f$ for all $f \in \clh_k$, and hence, $\vp \clh_k \subseteq \clh_k$. This proves that $A = M_\vp$, and completes the proof of the lemma.
\end{proof}

Now we prove the main result of this section. It essentially says that $\clm(\clh_k) = H^\infty(\D)$, that is, $\{M_z\}' = \{M_{\vp}: \vp \in H^\infty(\D)\}$.

\begin{Theorem}\label{thm: multiplier TDS}
Let  $\vp: \D \raro \C$ be a function, and let $\clh_k$ be a truncated space with $\{f_m\}_{m \geq 0}$ as an orthonormal basis, where $f_m(z) = (1 + b_m z) z^m$, $m \geq 0$, and $b_t = 0$ for all $t \geq n$. Then $\vp \in \clm(\clh_k)$ if and only in $\vp \in H^\infty(\D)$.
\end{Theorem}

\begin{proof}
Recall from \eqref{eq: M_z} that
\[
z f_m  = f_{m+1} + c_{m,1} \sum_{t=0}^\infty (-1)^t \Big(\prod_{j=0}^{t-1} b_{m+2+j}\Big) f_{m+2+t} \quad \quad (m\geq 0).
\]
In general, for any $p \geq 1$, we have
\[
z^p f_m = (1+b_m z)z^{m+p} = f_{m+p}+(b_m - b_{m+p})z^{m+p+1}.
\]
Since $c_{m,p}=b_m-b_{m+p}$ for all $m \geq 0$ and $p \geq 1$ (see \eqref{eq: c_mp}), it follows that
\begin{equation}\label{eqn:z^p f_m}
z^p f_m = f_{m+p}+c_{m,p}(f_{m+p+1}-b_{m+p+1}f_{m+p+2}+b_{m+p+1}b_{m+p+2}f_{m+p+3}-\cdots).
\end{equation}
Note that $c_{m,p}=0$ for all $m\geq n$. Let $\vp \in \clh_k$, and suppose $\vp = \sum_{m=0}^{\infty} \alpha_m z^m$. Since $\vp f_0 = \sum_{m=0}^{\infty} (\alpha_m z^m f_0)$ and $f_0 = 1 + b_0 z$, \eqref{eqn:z^p f_m} implies
\[
\vp f_0 = \alpha_0 f_0 + \alpha_1 f_1 + (\alpha_2 + \beta_{1,0}) f_2 + \cdots + (\alpha_n + \beta_{n-1,0}) f_n + \sum_{t=n+1}^{\infty} (\alpha_{t} + c_{0, t-1} \alpha_{t-1}) f_{t},
\]
where
\[
\beta_{j,0} = \text{coefficient of } f_{j+1} - \alpha_{j+1} \qquad (j=1, \ldots, n-1).
\]
Observe that $\beta_{j,0}$ is a finite sum for each $j=1, \ldots, n-1$. Similarly, for each $0 \leq m < n$, we have
\[
\begin{split}
\vp f_m = & \alpha_0 f_m + \alpha_1 f_{m+1} + (\alpha_2 + \beta_{1,m}) f_{m+2} + \cdots + (\alpha_{n-m} + \beta_{n-m-1,m}) f_n
\\
& \quad + \sum_{t=n+1}^{\infty} (\alpha_{t-m} + c_{m, t-m-1} \alpha_{t-m-1}) f_{t},
\end{split}
\]
where, as before, we let
\[
\beta_{j,m} = \text{coefficient of } f_{j+m+1} - \alpha_{j+1} \qquad (j=1, \ldots, n-m-1).
\]
Finally, for each $m \geq n$, it is easy to see that
\[
\vp f_m = \sum_{j=0}^\infty \alpha_j f_{m+j}.
\]
Therefore, the formal matrix representation of the linear operator $M_\vp$ (which is not necessarily bounded yet) is given by the formal sum of matrix operators
\begin{equation}\label{phi Matrix}
[M_{\vp}] = [\tilde{T}_{\vp}] + [N],
\end{equation}
where
\begin{equation}\label{eqn: matrix T phi}
[\tilde{T}_{\vp}] = \begin{bmatrix}
\alpha_0 & 0 & 0  & 0  &\dots
\\
\alpha_1 & \alpha_0 & 0 & 0 &\ddots
\\
\alpha_2 & \alpha_1 & \alpha_0  & 0 &\ddots
\\
\vdots & \ddots & \ddots  &\ddots & \ddots
\end{bmatrix}
\end{equation}
and
\begin{equation}\label{eqn: matrix M1}
[N] =  \begin{bmatrix}
0& 0& 0&\dots& 0 &0 &0 &\dots
\\
0& 0& 0& \dots &0 &0 &0 &\ddots
\\
\beta_{1,0}& 0& 0& \dots &0 &0 &0 &\ddots
\\
\beta_{2,0} &  \beta_{1,1} &0 &\dots &0 & 0 & 0 &\ddots
\\
\vdots & \vdots & \vdots & \vdots & \vdots &\vdots &\vdots &\ddots
\\
\beta_{n-1,0} & \beta_{n-2,1} & \beta_{n-3,2} & \dots &0 &0 &0 &\ddots
\\
c_{0,n} \alpha_n & c_{1,n-1} \alpha_{n-1}& c_{2,n-2} \alpha_{n-2}&\dots & c_{n-1,1} \alpha_1& 0& 0&\ddots
\\
c_{0,n+1} \alpha_{n+1} & c_{1,n} \alpha_n & c_{2,n-1} \alpha_{n-1}&\dots & c_{n-1,2} \alpha_2& 0 &0 & \ddots
\\
c_{0,n+2} \alpha_{n+2} & c_{1,n+1} \alpha_{n+1}& c_{2,n} \alpha_n &\dots &c_{n-1,3} \alpha_3& 0 &0 &\ddots
\\	
\vdots & \ddots & \ddots & \ddots & \ddots &\ddots &\ddots & \ddots
\\
\end{bmatrix}.
\end{equation}
Now assume that $\vp \in \clm(\clh_k)$, that is, the multiplication operator $M_{\vp}$ is bounded on $\clh_k$. Since
\[
M_\vp f_n = \vp f_n = \sum_{j=0}^{\infty}\alpha_j f_{n+j},
\]
it follows that $\{\alpha_m\}_{m\geq 0}$ is square summable, and hence $[\tilde{T}_{\vp}]$ defines a linear (but not necessarily bounded yet) operator on $\clh_k$. Since the matrix operator $[N]$ has at most $n$ nonzero columns and
\[
\sum_{m=0}^{\infty}|\alpha_m|^2<\infty,
\]
it follows that $[N]$ is bounded on $\clh_k$. Therefore, by \eqref{phi Matrix}, $[\tilde{T}_{\vp}]$ defines a bounded linear operator $\tilde{T}_{\vp}$ on $\clh_k$. Then we find that the canonical unitary map $U:\clh_k \raro H^2(\D)$ defined by equation \eqref{eqn:U map} satisfies
\[
U \tilde{T}_{\vp}=T_{\vp}U,
\]
where $T_{\vp}$ denote the (bounded) Toeplitz operator on $H^2(\D)$ with symbol $\vp$. In particular, $\vp \in H^{\infty}(\D)$.
	
For the converse, we assume that $\vp = \sum_{m=0}^\infty \alpha_m z^m$ is in $H^{\infty}(\D)$. If we set $\tilde{T}_{\vp} = U^* T_{\vp} U$, then $\tilde{T}_{\vp}$ is a bounded linear operator on $\clh_k$, and the matrix representation of $\tilde{T}_{\vp}$ will be of the form \eqref{eqn: matrix T phi}. Finally, since $\sum_{m=0}^\infty |\alpha_m|^2 < \infty$, it follows that the matrix \eqref{eqn: matrix M1} defines a bounded linear operator on $\clh_k$. Therefore, $M_{\vp} = \tilde{T}_{\vp} + N$ is bounded on $\clh_k$, which completes the proof of the theorem. 	
\end{proof}

Of course, the inclusion $\clm(\clh_k) \subseteq H^\infty(\D)$ follows rather trivially from properties of kernel functions: Suppose $\vp \in \clm(\clh_k)$. By the reproducing property of kernel functions, we have $M_{\vp}^* k(\cdot, w) = \overline{\vp(w)} k(\cdot, w)$, which implies
\[
|\vp(w)| = \frac{1}{\|k(\cdot, w)\|} \|M_{\vp}^* k(\cdot, w)\| \leq \|M_{\vp}\| \qquad (w \in \D).
\]
In particular, $\vp \in H^\infty(\D)$ and $\|\vp\|_{\infty} \leq \|M_{\vp}\|$. Evidently, the content of the above theorem is different and proves much more than the standard inclusion $\clm(\clh_k) \subseteq H^\infty(\D)$. Also, note that we have proved more than what has been explicitly stated in the above theorem:

\begin{Theorem}\label{thm: n shift commutator}
Consider the $n$-shift $S_n$ corresponding to the truncated space $\clh_k$ defined as in Theorem \ref{thm: multiplier TDS}, and let $X \in \clb(H^2(\D))$. Then $X \in \{S_n\}'$ if and only if there exists $\vp \in H^\infty(\D)$ such that $X = T_\vp + N$, where $N$ is a matrix operator as in \eqref{eqn: matrix M1} with respect to $\{z^m\}_{m \geq 0}$.
\end{Theorem}

The proof follows easily, once one observe that
\begin{equation}\label{eqn: U intertwine}
U M_{\vp} = (T_{\vp} + N) U,
\end{equation}
for all $\vp \in H^\infty(\D) = \clm(\clh_k)$, where $U: \clh_k \raro H^2(\D)$ is the canonical unitary as in \eqref{eqn:U map}.

The following observation is now standard: The $n$-shift $S_n$ as in Theorem \ref{thm: multiplier TDS} is irreducible. Indeed, if $\clm \subseteq \clh_k$ is a closed $M_z$-reducing subspace, then $P_{\clm} M_z =M_z P_{\clm}$ implies that $P_{\clm} = M_{\vp}$ for some $\vp \in \clm(\clh_k)$. By Theorem \ref{thm: multiplier TDS}, $\vp \in H^\infty(\D)$. Then $P_{\clm}^2 =P_{\clm}$ implies that $\vp^2 = \vp$ on $\D$, and we obtain $\vp \equiv 0$ or $1$. It now follows that $\clm = \{0\}$ or $\clh_k$.

Representations of commutants of $n$-shifts on even ``simple'' truncated spaces appear to be interesting and nontrivial. We will work out some concrete examples in Section \ref{sec: examples}.

\newsection{Hyperinvariant subspaces}\label{sec: Hyper}

We continue from where we left in Section \ref{sec: commutant}, and prove that invariant subspaces of $n$-shifts on truncated spaces are hyperinvariant. Recall that a closed subspace $\clm \subseteq \clh$ is called a \textit{hyperinvariant subspace} for $T \in \clb(\clh)$ if
\[
X \clm \subseteq \clm,
\]
for all $X \in \{T\}'$. We assume that $\clh_k$ is a truncated space corresponding to the orthonormal basis $\{f_m\}_{m \geq 0}$, where $f_m(z) = (1 + b_m z)z^m$, $m \geq 0$, and $\{b_m\}_{m\geq 0}$ are scalars such that $b_t = 0$ for all $t \geq n$. In this case, recall that $\clm(\clh_k) = H^\infty(\D)$ (see Theorem \ref{thm: multiplier TDS}), and the canonical unitary $U: \clh_k \raro H^2(\D)$ defined by equation \eqref{eqn:U map} satisfies
\[
U M_z = S_n U \text{ and } U M_\vp = (T_\vp + N)U,
\]
for all $\vp \in H^\infty(\D)$, where $N$ is the finite rank operator whose matrix representation with respect to the orthonormal basis $\{z^m\}_{m \geq 0}$ of $H^2(\D)$ is given by \eqref{eqn: matrix M1}.

We are now ready to solve the hyperinvariant subspace problem for $n$-shifts on truncated spaces.

\begin{Theorem}\label{thm: hyper}
Closed invariant subspaces of $n$-shifts on truncated spaces are hyperinvariant.
\end{Theorem}

\begin{proof}
Let $M_z$ be an $n$-shift on a truncated space, and let $S_n$ be the corresponding $n$-shift on $H^2(\D)$. Suppose $\clm$ is a nonzero closed $S_n$-invariant subspace of $H^2(\D)$. By Theorem \ref{thm: inv sub}, there exist an inner function $\theta \in H^\infty(\D)$ and polynomials $\{p_i, q_i\}_{i=0}^{n-1}$ such that
\[
\clm = (\C \vp_0 \oplus \C \vp_1 \oplus \cdots \oplus \C \vp_{n-1}) \oplus z^n\theta H^2(\D),
\]
where $\vp_i = z^i p_i \theta - q_i$ for all $i=0, \ldots, n-1$, and $S_n \vp_j \in (\mathbb{C}\vp_{j+1} \oplus \cdots \mathbb{C}\vp_{n-1})\oplus z^n\theta H^2(\D)$ for all $j=0, \ldots, n-2$, and $S_n \vp_{n-1} = z^n p_{n-1} \theta$. In view of Theorem \ref{thm: n shift commutator}, we only need to prove that $(T_\vp + N) \vp_i \in \clm$ for all $i=0, 1, \ldots, n-1$, and $(T_\vp + N)z^n\theta H^2(\D) \subseteq z^n\theta H^2(\D)$ for all $\vp \in H^\infty(\D)$. To this end, let $\vp \in \clm(\clh_k) = H^\infty (\D)$, and suppose $\vp(z) = \sum_{m=0}^{\infty} \alpha_m z^m$. Then for each $i=0, 1, \ldots, n-1$, we have
\[
(T_\vp + N)\vp_i = U M_\vp U^* \vp_i = U(\vp  U^*\vp_i),
\]
and hence
\[
\begin{split}
(T_\vp + N)\vp_i & = U(\sum_{m=0}^{\infty} \alpha_m z^m U^* \vp_i)
\\
& = U(\sum_{m=0}^{\infty} \alpha_m M_z^m U^* \vp_i)
\\
& = \sum_{m=0}^{\infty} \alpha_m S_n^m \vp_i \in \clm
\end{split},
\]
as $\vp_i \in \clm$ and $S_n \clm \subseteq \clm$. Finally, if $f \in H^2(\D)$, then Lemma \ref{lemma: S_n is shift} implies
\[
(T_\vp + N) z^n\theta f = T_\vp (z^n\theta f) + 0 = z^n \theta \vp f \in z^n\theta H^2(\D),
\]
and hence, $(T_\vp + N)z^n\theta H^2(\D) \subseteq z^n\theta H^2(\D)$, which completes the proof.
\end{proof}

Now let $M_z$ be an $n$-shift, and let $\clm(\clh_k) = H^\infty(\D)$. In particular, $\{M_z\}' = H^\infty(\D)$. In this case, a similar argument as the above proof gives the same conclusion as  Theorem \ref{thm: hyper}. However, as is well known, explicit computation of $\clm(\clh_k)$ is a rather challenging problem.

\section{Examples}\label{sec: examples}

In this section, we examine Theorem \ref{thm: inv sub} from a more definite examples point of view. As we will see, these examples are instructive and bring out several analytic and geometric flavors, and points out additional complications to the theory of finite rank perturbations.  

Fix scalars $a_0$ and $b_0$ such that $0 < |b_0| \leq |a_0|$, and consider the $1$-shift $S_1 = M_z + F$ on $H^2(\D)$ corresponding to the $1$-perturbation
\begin{equation}\label{eqn: def of F}
F z^m = \begin{cases}
((a_0 - 1) + b_0 z) z & \mbox{if } m=0 \\
0 & \mbox{if } m \geq 1.
\end{cases}
\end{equation}
The fact that $S_1$ is a $1$-shift follows from the inherited tridiagonal structure of $S_1$. Indeed, $S_1$ is unitarily equivalent to the shift $M_z$ on the truncated space $\clh_k$ with orthonormal basis $\{f_m\}_{m\geq 0}$, where $f_m = (a_m + b_mz) z^m$, $m\geq 0$, and $a_t = 1$ and $b_t=0$ for all $t \geq 1$. Since
\[
\Big|\frac{a_m}{a_{m+1}}\Big| \geq \min \{|a_0|,1\} \qquad (m \geq 0) ,
\]
the sequence $\{|\frac{a_m}{a_{m+1}}|\}_{m \geq 0}$ is bounded away from zero, and hence, $M_z$ is left-invertible (see the discussion following Definition \ref{def: trunc}). Moreover, the canonical unitary  $U: \clh_k \raro H^2(\D)$ defined by equation \eqref{eqn:U map} satisfies the required intertwining property $U M_z=S_1 U$. Therefore, it follows that $S_1= M_z + F$ on $H^2(\D)$ is indeed a $1$-shift. We clearly have
\begin{equation}\label{eqn: F f=}
F f = f(0) ((a_0 - 1) + b_0z) z \qquad (f \in H^2(\D)).
\end{equation}
Now we observe three distinctive features of $S_1$: Note that the matrix representation of $S_1$ with respect to the orthonormal basis $\{z^m\}_{m \geq 0}$ of $H^2(\D)$ is given by
\[
[S_1] = [M_z + F] = \begin{bmatrix}
0 & 0 & 0  & 0 & \dots
\\
a_0 & 0 & 0 & 0 & \ddots
\\
b_0 & 1 & 0  & 0 & \ddots
\\
0 & 0 & 1 & 0 &  \ddots
\\
\vdots & \ddots & \ddots & \ddots & \ddots
\end{bmatrix}.
\]
Then, a simple computation yields that
\[
[S_1^* ,S_1] =
\begin{bmatrix}
|a_0|^2 + |b_0|^2 & \bar{b}_0 & 0  & 0 & \dots
\\
b_0 & 1 - |a_0|^2 & -a_0 \bar{b}_0 & 0 & \ddots
\\
0 & -\bar{a}_0 b_0 & -|b_0|^2  & 0& \ddots
\\
0 & 0 & 0 & 0 & \ddots
\\
\vdots & \ddots & \ddots&\ddots & \ddots
\end{bmatrix},
\]
is precisely a rank-$3$ operator. Indeed, the determinant of the $3 \times 3$ nonzero submatrix of $[S_1^* ,S_1]$ is given by
\[
(|a_0|^2+|b_0|^2)\Big(-(1-|a_0|^2) |b_0|^2 - |a_0|^2|b_0|^2 \Big)- |b_0|^4 = -|a_0|^2|b_0|^2 <0.
\]
This also implies that $[S_1^* ,S_1]$ is not a positive definite operator. Therefore:
\begin{enumerate}
\item $S_1$ is essentially normal, that is, $[S_1^*, S_1] = S_1^* S_1 - S_1 S_1^*$ is compact (in fact, here it is of finite rank).
\item $S_1$ is not hyponormal (and hence, not subnormal).
\item Invariant subspaces of $S_1$ are cyclic.
\end{enumerate}

The proof of the final assertion is the main content of the following two theorems:

\begin{Theorem}\label{thm: cyclic}
Let $a_0$ and $b_0$ be scalars such that $0 < |b_0| \leq |a_0|$. Suppose
\[
F z^m = \begin{cases}
((a_0 - 1) + b_0 z) z & \mbox{if } m=0 \\
0 & \mbox{if } m \geq 1,
\end{cases}
\]
and consider the $1$-shift $S_1 = M_z + F$ on $H^2(\D)$. Then a nonzero closed subspace $\clm \subseteq H^2(\D)$ is invariant under $S_1$ if and only if there exists an inner function $\theta \in H^\infty(\D)$ such that
\[
\clm= \mathbb{C}\vp \oplus z\theta H^2(\D),
\]
where
\[
\vp = \Big(1 + \frac{b_0}{a_0}|\theta(0)|^2z \Big) \theta - \frac{\theta(0)}{a_0} \Big((a_0-1) + b_0z \Big).
\]
Moreover, if $\clm$ is as above, then $\clm=[\vp]_{S_1}$.
\end{Theorem}

\begin{proof}
In view of Theorem \ref{thm: inv sub}, we only have to prove the necessary part. Suppose $\clm$ is a $S_1$-invariant closed subspace of $H^2(\D)$. Again, by Theorem \ref{thm: inv sub}, there exists inner function $\theta \in H^\infty(\D)$ such that $\clm= \mathbb{C}\vp\oplus z\theta H^2(\D)$, where $S_1 \vp = z p \theta$ and
\begin{equation}\label{eqn: phi = p + qtheta}
\vp = q + p \theta
\end{equation}
for some polynomials $p, q \in \C[z]$.  Since $S_1 \vp = z p \theta$, we have $z p \theta = (M_z + F) \vp$. Then \eqref{eqn: F f=} implies
\[
z p \theta = (M_z + F) \vp = z \vp + \vp(0) ((a_0 - 1) + b_0 z) z,
\]
that is, $p \theta = \vp + \vp(0) ((a_0 - 1) + b_0 z)$. Therefore,
\begin{equation}\label{eqn: phi = p theta + }
\vp = p \theta - \vp(0) ((a_0 - 1) + b_0 z),
\end{equation}
and by \eqref{eqn: phi = p + qtheta}, it follows that $q = - \vp(0) ((a_0 - 1) + b_0 z)$. Now, if $m \geq 1$, then $\vp \perp z^m \theta H^2(\D)$ implies that $\langle \vp, z^m \theta \rangle = 0$, and hence \eqref{eqn: phi = p theta + } yields
\[
\langle p, z^m \rangle = \langle p \theta, z^m \theta \rangle = \vp(0) \langle (a_0-1) + b_0 z, z^{m} \theta \rangle.
\]
Since $\vp(0) = \frac{p(0) \theta(0)}{a_0}$, by \eqref{eqn: phi = p theta + } again, it follows that
\[
\langle p, z^m \rangle = \begin{cases}
b_0 \frac{p(0) |\theta(0)|^2}{a_0} & \mbox{if } m=1 \\
0 & \mbox{if } m > 1.
\end{cases}
\]
Thus, we have
\[
p = p(0) (1 + \frac{b_0}{a_0} |\theta(0)|^2 z),
\]
which implies that (by recalling \eqref{eqn: phi = p theta + })
\[
\begin{split}
\vp & = p \theta - \vp(0) ((a_0 - 1) + b_0 z)
\\
& = p \theta - \frac{p(0)\theta(0)}{a_0} ((a_0 - 1) + b_0 z)
\\
& = p(0) \Big[(1 + \frac{b_0}{a_0} |\theta(0)|^2 z) \theta - \frac{\theta(0)}{a_0} ((a_0 -1) + b_0 z)\Big].
\end{split}
\]
Finally, since $\vp \neq 0$, without loss of generality, we may assume that $p(0) = 1$. This completes the proof of the first part. We also have
\[
p = 1 + \frac{b_0}{a_0} |\theta(0)|^2 z.
\]
Since $0 < |b_0| \leq |a_0|$ and $\theta$ is inner, it follows that $p$ is an outer polynomial. The remaining part of the statement is now a particular case of the following theorem.
\end{proof}

In the level of $S_1$-invariant subspaces, we have the following general classification:

\begin{Theorem}\label{thm: S_1 wandering}
Let $\clm \subseteq H^2(\D)$ be a nonzero closed $S_1$-invariant subspace. Then
\[
\clm = [\clm \ominus S_1 \clm]_{S_1},
\]
if and only if there exists an inner function $\theta \in H^\infty(\D)$ and an outer polynomial $p \in \C[z]$ such that $\clm \ominus S_1 \clm = \C \vp$ and $S_1 \vp = z p \theta$.
\end{Theorem}

\begin{proof}
Let $\clm = \C \vp \oplus z \theta H^2(\D)$, where $\theta \in H^\infty(\D)$ is an inner function, $\vp = p \theta - q$, and $S_1 \vp = z p \theta$ for some $p,q \in \C[z]$ (see Theorem \ref{thm: inv sub}). Note that
\[
\clm \ominus S_1 \clm = \C \vp.
\]
Since $S_1 \vp = z p \theta$, by \eqref{eqn: Smf= Mzlm} we have
\[
S_1^m \vp = S_1^{m-1}(z p \theta) = M_z^{m-1} (z p \theta) = z^m p \theta,
\]
for all $m \geq 2$. Therefore
\begin{equation}\label{eqn: S_1 m vp}
S_1^m \vp = z^m p \theta \qquad (m \geq 1).
\end{equation}
Now suppose that $\clm = [\vp]_{S_1}$. The above equality then tells us that $[S_1 \vp]_{S_1} \subseteq z \theta H^2(\D)$. Since $\vp \perp z \theta H^2(\D)$, we have
\[
\clm = [\vp]_{S_1} = \C \vp \oplus z \theta H^2(\D) = \C \vp \oplus [S_1 \vp]_{S_1}.
\]
Clearly, we have $[S_1 \vp]_{S_1} = z \theta H^2(\D)$, where on the other hand
\[
[S_1 \vp] = [z p \theta]_{M_z} = z \theta [p]_{M_z},
\]
and hence $z \theta [p]_{M_z} = z \theta H^2(\D)$. But since $z \theta$ is an inner function, we have $[p]_{M_z} = H^2(\D)$, that is, $p$ is an outer polynomial. In the converse direction, since $p$ is outer, \eqref{eqn: S_1 m vp} implies that
\[
z \theta H^2(\D) = z \theta [p]_{M_z} = [S_1 \vp]_{M_z} = [S_1 \vp]_{S_1}.
\]
Therefore
\[
\clm = \C \vp \oplus z \theta H^2(\D) = \C \vp \oplus [S_1 \vp]_{S_1} = [\vp]_{S_1},
\]
which completes the proof of the theorem.
\end{proof}

In the setting of Theorem \ref{thm: cyclic}, we now consider the particular case when $a_0 = b_0 = 1$. In this case
\[
F z^m = \begin{cases}
z^2 & \mbox{if } m=0 \\
0 & \mbox{if } m \geq 1.
\end{cases}
\]
Then, by Theorem \ref{thm: cyclic}, we have:

\begin{Corollary}\label{cor: cyclic}
Let $F1 = z^2$ and $F z^m = 0$ for all $m \geq 1$. Suppose $\clm$ is a nonzero closed subspace of $H^2(\D)$. Then $\clm$ is invariant under $S_1 = M_z + F$ if and only if there exists an inner function $\theta \in H^\infty(\D)$ such that $\clm= \mathbb{C}\vp \oplus z\theta H^2(\D)$, where
\[
\vp = (1 + |\theta(0)|^2z) \theta - \theta(0)z.
\]
Moreover, if $\clm$ is as above, then $\clm=[\vp]_{S_1}$.
\end{Corollary}

Moreover, in the setting of Theorem \ref{thm: cyclic}, for $\clm=\mathbb{C}\vp\oplus z\theta H^2(\D)$, we have the following curious observations:

\begin{enumerate}
\item $\clm$ is of finite codimension if and only if $\theta$ is a finite Blaschke product (this is also true for general $n$-shift invariant subspaces in the setting of Theorem \ref{thm: inv sub}).
\item $\vp$ need not be an inner function. Indeed, in the setting of Corollary \ref{cor: cyclic}, consider the Blaschke factor $\theta(z)=\frac{\frac{1}{2}-z}{1-\frac{1}{2}z}$, and set $\vp = (1 + |\theta(0)|^2z) \theta - \theta(0)z$. Then $\vp(z)=\frac{1}{2} \frac{1-\frac{11}{4}z}{1-\frac{1}{2}z}$ is a rational function with $z = 2$ as the only pole. Note that $\vp(1) = - \frac{7}{4}$ and $\vp(-1) = \frac{5}{4}$. Clearly, $\vp$ is not an inner function.

\item If $\theta(0) = 0$, then $\clm = [\theta]_{M_z} = [\theta]_{S_1}$. Therefore, $S_1|_{\clm}$ is an unilateral shift of multiplicity one. On the other hand, if $\tilde{\theta}$ is an inner function with $\tilde\theta(0) \neq 0$, then $S_1|_{\clm}$ and $S_1|_{\tilde{\clm}}$ are not unitarily equivalent, where $\tilde{\clm} = \C \tilde{\vp} \oplus z \tilde{\theta} H^2(\D)$ and $\tilde{\vp} = (1 + \frac{b_0}{a_0}|\tilde{\theta}(0)|^2z) \tilde{\theta} - \frac{\tilde{\theta}(0)}{a_0} \Big((a_0-1) + b_0z \Big)$.
\end{enumerate}

The final observation is in sharp contrast with a well-known consequence of the Beurling theorem: If $\clm_1$ and $\clm_2$ are nonzero closed $M_z$-invariant subspaces of $H^2(\D)$, then $M_z|_{\clm_1}$ and $M_z|_{\clm_2}$ are unitarily equivalent. In view of (3) above, this property fails to hold for invariant subspaces of $n$-shifts.

We still continue with the setting of Theorem \ref{thm: cyclic}, and examine Theorem \ref{thm: n shift commutator} in the case of the commutators of $S_1$. In fact, we have the following observation: Let $X \in \clb(H^2(\D))$. Then $X \in \{S_1\}'$ if and only if there exists $\vp \in H^\infty(\D)$ such that $X = T_{\vp} + N$, where
\[
Nz^m =
\begin{cases}
z (\vp - \vp(0)) & \mbox{if } m=0
\\
0 & \mbox{otherwise}.
\end{cases}
\]
Indeed, in this case, $f_0(z)=1+z$ and $f_m(z) = z^m$ for all $m\geq 1$. Let $X \in \clb(H^2(\D))$, and let $X S_1 = S_1 X$. Set $\tilde{X} = U^* X U$. Then, $\tilde{X} \in \clb(\clh_k) \cap \{M_z\}'$, and, as in the proof of Theorem \ref{thm: n shift commutator}, there exist $\vp \in H^\infty(\D)$ such that $\tilde{X} = M_{\vp}$. Moreover, if $\vp = \sum_{m=0}^\infty \alpha_m z^m$, then
\[
M_{\vp} f_0 = \alpha_0 f_0 + \alpha_1 f_1 + \sum_{j=2}^\infty (\alpha_j + \alpha_{j-1}) f_j,
\]
and
\[
M_{\vp} f_m = \sum_{j=0}^\infty \alpha_j f_{m+j} \qquad (m \geq 1),
\]
which implies that
\[
[M_{\vp}] = \begin{bmatrix}
\alpha_0 & 0 & 0  & 0 & \cdots
\\
\alpha_1 & \alpha_0 & 0 & 0 & \ddots
\\
\alpha_2 + \alpha_1 & \alpha_1 & \alpha_0  & 0 & \ddots
\\
\alpha_3 + \alpha_2 & \alpha_2 & \alpha_1 & \alpha_0 & \ddots
\\
\vdots & \ddots & \ddots & \ddots & \ddots
\end{bmatrix}.
\]
Therefore, $[M_\vp] = [\tilde{T}_{\vp}] + [N]$, where
\[
[\tilde{T}_{\vp}] = \begin{bmatrix}
\alpha_0 & 0 & 0  & 0 & \cdots
\\
\alpha_1 & \alpha_0 & 0 & 0 & \ddots
\\
\alpha_2 & \alpha_1 & \alpha_0  & 0 & \ddots
\\
\alpha_3 & \alpha_2 & \alpha_1 & \alpha_0 & \ddots
\\
\vdots & \ddots & \ddots & \ddots & \ddots
\end{bmatrix}
\text{ and }
[N] = \begin{bmatrix}
0 & 0 & 0  & 0 & \cdots
\\
0 & 0 & 0 & 0 & \ddots
\\
\alpha_1 & 0 & 0  & 0 & \ddots
\\
\alpha_2 &0 & 0 & 0 & \ddots
\\
\vdots & \ddots & \ddots & \ddots & \ddots
\end{bmatrix}.
\]
By the proof of Theorem \ref{thm: n shift commutator}, $X = U \tilde{X} U^* = T_{\vp} + N$. Clearly, $N1 = \sum_{j=1}^\infty \alpha_j z^{j+1} = z ( \vp - \vp(0))$, and $N z^m = 0$ for all $m \geq 1$, which ends the proof of the claim.

In connection with Theorem \ref{thm: cyclic}, we now point out the other natural (but easier) example of $1$-shift $S_1 = M_z + F$, where
\[
F z^m = \begin{cases}
z & \mbox{if } m=0 \\
0 & \mbox{if } m \geq 1.
\end{cases}
\]
In this case, $S_1$ is a weighed shift with the weight sequence $\{2,1,1, \ldots\}$. Therefore, $S_1$ is similar to the unilateral shift $M_z$ on $H^2(\D)$ via an explicit similarity map. Using this, it is rather easy to deduce, by pulling back inner functions corresponding to $M_z$-invariant subspaces of $H^2(\D)$, that $S_1$-invariant subspaces are cyclic and of the form $\C \vp \oplus z \theta H^2(\D)$, with $\theta \in H^\infty(\D)$  inner and (after an appropriate scaling)
\[
\vp = \theta - \frac{1}{2} \theta(0).
\]
We refer to \cite{Nikolskii} for the theory of invariant subspaces of weighted shifts.

Finally, as far as the results of this present paper are concerned, $n$-shifts are more realistic shifts among shifts that are finite rank perturbations of the unilateral shift. However, a pressing question remains about the classification of invariant subspaces of general shifts that are finite rank perturbations of the unilateral shift.

\vspace{0.1in}

\noindent\textbf{Acknowledgement:}
The research of the second named author is supported in part by NBHM grant NBHM/R.P.64/2014, and the Mathematical Research Impact Centric Support (MATRICS) grant, File No: MTR/2017/000522 and Core Research Grant, File No: CRG/2019/000908, by the Science and Engineering Research Board (SERB), Department of Science \& Technology (DST), Government of India.

\bibliographystyle{amsplain}

\end{document}